\documentclass[12pt]{article}
\usepackage{mathrsfs}
\usepackage{t1enc}
\usepackage[latin1]{inputenc}
\usepackage[english]{babel}
\usepackage{amsmath,amsthm,amssymb}
\usepackage{amsfonts}
\usepackage{latexsym}
\usepackage{enumerate}

\usepackage{graphicx}
\usepackage[natural]{xcolor}
\usepackage{tikz}
\usepackage{tikz-3dplot}
\usetikzlibrary{shapes}
\usetikzlibrary{arrows,decorations.pathmorphing,backgrounds,positioning,fit,petri,automata}
\usetikzlibrary{positioning}

\theoremstyle{plain}
\newtheorem{thm}{Theorem}

\newtheorem{lem}[thm]{Lemma}

\newtheorem{conj}[thm]{Conjecture}

\theoremstyle{plain}

\theoremstyle{plain}

\theoremstyle{plain}

\title{Antimagic orientation of lobsters
\thanks{This work is supported by NSFC (11901263, 61802158)}}
\author{Yuping Gao$^{a}$
\thanks{Corresponding author. \emph{E-mail address:} gaoyp@lzu.edu.cn.},
Songling Shan$^{b}$\\
{\small a. School of Mathematics and Statistics, Lanzhou University, Lanzhou 730000, China}\\
{\small b. Department of Mathematics, Illinois State University, Normal, IL 61790, USA}}
\date{}

\begin{document}
\baselineskip 0.65cm

\maketitle
\begin{abstract}

Let $m\ge 1$ be an integer and $G$ be a graph with $m$ edges. We say that $G$ has an antimagic orientation if $G$ has an orientation $D$ and a bijection $\tau:A(D)\rightarrow \{1,2,\cdots,m\}$ such that no two vertices in $D$ have the same vertex-sum under $\tau$, where the vertex-sum of a vertex $u$ in $D$ under $\tau$ is the sum of labels of all arcs entering $u$ minus the sum of labels of all arcs leaving $u$. Hefetz, M\"{u}tze and Schwartz [J. Graph Theory, 64: 219-232, 2010] conjectured that every connected graph admits an antimagic orientation.
The conjecture was confirmed for certain classes of graphs such as dense graphs, regular graphs, and trees including caterpillars and  $k$-ary trees.
In this note, we prove that every lobster admits an antimagic orientation.

\end{abstract}

\section{Introduction}
Let $q>p$  be two integers. We use $[p,q]$ to denote the set $\{p,p+1,\cdots,q\}$,
and simply write $[1,q]$ as $[q]$.
Graphs considered in this paper are simple, and  digraphs in consideration are those which are orientations of simple graphs.
   Let $m\ge 1$ be an integer and $G$ be a graph with $m$ edges. We use $V(G),E(G),\Delta(G)$ and $\delta(G)$ to denote the vertex set, the edge set, the maximum degree and the minimum degree of $G$, respectively.  An \emph{antimagic labeling} of $G$ is a bijection $\tau$ from $E(G)$ to $[m]$ such that for any two distinct vertices $u$ and $v$ in $G$, the sum of labels on the edges incident with $u$ differs from that of $v$. A graph is said to be \emph{antimagic} if it admits an antimagic labeling. The concept of antimagic labeling was proposed by Hartsfield and Ringel in 1990 \cite{HR1990}. In the same paper, they conjectured that every connected graph other than $K_2$ is antimagic and every tree other than $K_2$ is antimagic. This  topic was investigated by many researchers, for instance, see \cite{CLPZ2016,LZ2014, YHYWW2018}.
   One of the best known results for trees was due to Kaplan, Lev, and Roditty \cite{KLR2009}, who proved that any tree having more than two vertices and at most one vertex of degree two
is antimagic (see also \cite{LWZ2014}).
Lozano, Mora and Seara \cite{LMS2019} proved that caterpillars are antimagic, where a \emph{caterpillar} is a tree of order at least 3 such that the removal of its leaves produces a path. For related results the readers are referred to the
survey of Gallian \cite{G2017}.

In 2010, Hefetz, M\"{u}tze and Schwartz introduced a variation of antimagic labeling, i.e., antimagic labeling of a digraph \cite{HMS2010}. Let $D$ be a digraph. We use $V(D)$ and $A(D)$ to denote the set of vertices and the set of arcs of $D$, respectively. Let $X,Y\subseteq V(D)$ be two subsets, we use $A(X,Y)$ to denote the set of arcs with tail in $X$ and head in $Y$. The notation $A(X,\overline{X})$ is also denoted by $\partial(X)$. Let $|A(D)|=m$. An \emph{antimagic labeling} of $D$ is a bijection $\tau$ from $A(D)$ to $[m]$ such that no two vertices receive the same vertex-sum, where the \emph{vertex-sum} of a vertex $u\in V(D)$ is the sum of labels of all arcs entering $u$ minus the sum of labels of all arcs leaving $u$. We use $s_D(u)$ to denote the vertex-sum of the vertex $u\in V(D)$,
and simply write $s(u)$ if $D$ is understood.
 We say $(D,\tau)$ is an \emph{antimagic orientation} of a graph $G$ if $D$ is an orientation of $G$ and $\tau$ is an antimagic labeling of $D$. Hefetz, M\"{u}tze and Schwartz~\cite{HMS2010} proposed the following conjecture.

\begin{conj} \label{orientation-conj} Every connected graph admits an antimagic orientation.
\end{conj}

In the same paper, Hefetz, M\"{u}tze and Schwartz proved that Conjecture \ref{orientation-conj} is affirmative for some classes of %
graphs, such as  stars, wheels, cliques, and dense graphs (graphs of order $n$ with minimum degree at least $C\log n$ for an absolute
constant $C$); in fact, the authors proved a stronger result that every orientation of these graphs is antimagic. Recently, Conjecture \ref{orientation-conj} has been verified  for regular graphs \cite{HMS2010,LSWYZ2019,Y2019,SH2019}, biregular bipartite graphs with minimum degree at least two \cite{SY2017}, Halin graphs \cite{YCZ2019}, and graphs with large maximum degree \cite{YCOP2019}. Trees are widely investigated for graph labeling problems. For antimagic orientation,  it is proved that Conjecture \ref{orientation-conj} is true for caterpillars \cite{L2018} and  complete $k$-ary trees \cite{SH2019tree}.

It is easy to observe that all antimagic
bipartite graphs admit an antimagic orientation where all edges are oriented
in the same direction between the partite sets. Therefore, all subclasses
of trees that are known to be antimagic admit antimagic orientations.
A \emph{lobster}  is a tree such that the removal of its leaves produces a caterpillar.
It is still unknown that if every
lobster is antimagic.
In this paper, we prove that Conjecture \ref{orientation-conj} is affirmative for every lobster.

\begin{thm}\label{Thm} Every lobster  admits an antimagic orientation.
\end{thm}

\section{Preliminary Lemmas}

In this section, we prove two lemmas that will be used in proving Theorem~\ref{Thm}.

%

\begin{lem}\label{lemma3}  Every path with at least two edges is antimagic and every path admits an antimagic orientation.
\end{lem}

\begin{proof}
	Since $K_2$ admits an antimagic orientation, and every antimagic bipartite graph admits an antimagic orientation, it suffices to prove the first part of the statement.
	Let $m\ge 2$ be an integer and  $P=v_0v_1\cdots v_m$ be a path. We define a labeling of $P$ as follows.
	
	If  $m$ is even, starting from the edge incident to $v_0$
	consecutively assign labels
	$$
	1,2,3, \cdots, m-1, m.
	$$
	
	%
	%
	%
	
	If $m$ is odd,
	starting from the edge incident to $v_0$
	consecutively assign labels
	$$
	1,2,3, \cdots, m-2,  m, m-1.
	$$
	
	It is clear that the assignment is an antimagic labeling of $P$.
\end{proof}

Let $P$ be a path and $u,v\in V(P)$. We denote by $P[u,v]$  the subpath of $P$ with one end $u$ and the other end $v$.   The following lemma will be used for giving an antimagic orientation for the central path of a lobster, where vertices in $U$ are corresponding to the vertices of degree at least three in the lobster.

\begin{lem}\label{lemma1} Let $m\ge 2$ be an integer,  $P=v_0v_1\cdots v_m$ be a path, and $U=\{v_{h_1}, v_{h_2},\cdots,v_{h_t}\}\subseteq V(P)$ with $|U|\ge 1$, where $h_1<h_2<\cdots< h_t$. Then $P$ has an orientation $\overrightarrow{P}$ and a labeling $\tau:A(\overrightarrow{P}) \rightarrow [m]$ such that
\begin{enumerate}[$(i)$]
	\item $s(v)\geq1$ for any vertex $v\in U$;
	\item  $1\leq |s(v)|\leq m$ for every vertex $v\in V(P)\setminus U$; and
	\item $s(u)\neq s(v)$ for any two distinct vertices $u,v\in V(P)\setminus U$.
\end{enumerate}
\end{lem}

\begin{proof}
	We orient $P$ so that  $P[v_0,v_{h_t}]$ is a directed path from $v_0$ to $v_{h_t}$ and
	$P[v_m,v_{h_t}]$ is a directed path from $v_m$ to $v_{h_t}$. Denote the orientation by
	 $\overrightarrow{P}$. Define
$$P_0=P[v_0,v_{h_1}], \quad  P_i=P[v_{h_i}, v_{h_{i+1}}] \quad \text{for each} \quad i\in [t-1], \quad P_{t}=P[v_{h_t}, v_m]$$
to be the set of $t+1$ subpaths of $P$ separated by vertices in $U$.

By reversing  the orientation and the names of the vertices of $P$ if necessary,  we  assume that $h_1\ge m-h_t$
or $|V(P_0)|\ge |V(P_t)|$.
Let $$\ell=m-h_t \quad \text{and}  \quad s=h_1-\ell.$$ We label $P$ according to the parities of $\ell$ and $s$.

\begin{enumerate}[Step 1]
	\item Assign labels to edges  of  $P_t$ starting from the edge incident to $v_m$ consecutively using respectively the following labels:
	
	$
	\begin{cases}
	m-1, 2, m-3, 4, \cdots, m-\ell+1, \ell, & \text{if  $\ell$ is even};\\
	m-1, 2, m-3, 4, \cdots, m-\ell+2, \ell-1, m-\ell, & \text{if  $\ell$ is odd}.
	\end{cases}
	$
	
	\item
	Assign labels to edges  of  $P_0$ starting from the edge incident to $v_0$ consecutively using respectively the following labels:	
		$$
	\begin{cases}
	m, 1, m-2, 3, \cdots, m-\ell+2, \ell-1, m-\ell, & \text{if  $\ell$ is even and $s=1$};\\
m, 1, m-2, 3, \cdots, m-\ell+2, \ell-1, m-\ell,\\ \ell+1, m-\ell-1, \ell+2,m-\ell-2,  \cdots,\\ \ell+\lfloor \frac{s}{2}\rfloor, m-\ell-\lfloor \frac{s}{2}\rfloor, & \text{if  $\ell$ is even, $s$ is odd and $s\ge 3$};\\
 m, m-2, & \text{if  $\ell=2$  and $s=0$};\\
m, 1, m-2, 3, \cdots, \ell-3, m-\ell+2, m-\ell, & \text{if  $\ell\ge 4$ is even and $s=0$};\\
 m, 1, m-2, 3, \cdots, \ell-3, m-\ell+2, \ell-1,\\ m-\ell,  m-\ell-1, & \text{if  $\ell$ is even and $s=2$};\\
 m, 1, m-2, 3, \cdots, m-\ell+2, \ell-1, m-\ell,\\ \ell+1, m-\ell-1,  \ell+2, m-\ell-2, \cdots,\\ \ell+ \frac{s}{2} -1, m-\ell- \frac{s}{2}+1, m-\ell- \frac{s}{2}, & \text{if  $\ell$ is even, $s$ is even and $s\ge 4$};\\
 m, 1, m-2, 3, \cdots, m-\ell+3, \ell-2, m-\ell+1, & \text{if  $\ell$ is odd and $s=0$};\\
 m, 1, m-2, 3, \cdots, m-\ell+3, \ell-2, m-\ell+1,\\ \ell, m-\ell-1, & \text{if  $\ell$ is odd and $s=2$};\\
 m, 1, m-2, 3, \cdots, m-\ell+3, \ell-2, m-\ell+1,\\ \ell, m-\ell-1, \ell+1,  m-\ell-2, \cdots,\\ \ell+\frac{s}{2}-1, m-\ell- \frac{s}{2}, & \text{if  $\ell$ is odd, $s$ is even and $s\ge 4$};\\
 m, m-2, & \text{if  $\ell=1$ and $s=1$};\\
  m, 1, m-2, 3, \cdots, \ell-4, m-\ell+3, \ell-2,\\ m-\ell+1, m-\ell-1, & \text{if  $\ell\ge 3$ is odd and $s=1$};\\
 m, 1, m-2, 3, \cdots, m-\ell+3, \ell-2, m-\ell+1,\\ \ell, m-\ell-1, m-\ell-2, & \text{if  $\ell$ is odd and $s=3$};\\
  m, 1, m-2, 3, \cdots, m-\ell+3, \ell-2, m-\ell+1, \\ \ell, m-\ell-1, \ell+1, m-\ell-2, \cdots, \ell+\lfloor \frac{s}{2}\rfloor-1, \\  m-\ell-\lfloor \frac{s}{2}\rfloor, m-\ell-\lfloor \frac{s}{2}\rfloor-1, & \text{if  $\ell$ is odd, $s$ is odd and $s\ge 5$}.
	\end{cases}
	$$
\item 	For each  $i\in [t-2]$, starting from 1 consecutively, assign labels to edges of $P_i$ in the following way:
starting from the edge closest to $v_m$  consecutively using alternatively the largest available label
and the  smallest available label.

\item For the path $P_{t-1}$, if it has only one edge, assign the last available label to the edge.
Thus, we assume that $P_{t-1}$ has at least two edges.
We assign labels to edges of $P_{t-1}$ in the same fashion as in Step 3 unless:
(a) $\ell$ is even, $s \ge 2$ is even, and $|E(P_{t-1})|$ is odd; (b)  $\ell$ is odd, $s$ is odd and $s\ge 3$, and $|E(P_{t-1})|$ is odd, or (c) $\ell$ is even and $s=0$ or $\ell$ is odd and $s=1$,   and $|E(P_{t-1})|\ge 4$ is even.
In all these cases, assign labels  starting from the edge closest to $v_0$ consecutively using alternatively   the largest available label
and the  smallest available label.

\end{enumerate}
See Figure~\ref{fig1} for an illustration of the labeling given above. Denote the labeling described above by $\tau$. We show that $\tau$
satisfies the three properties in Lemma~\ref{lemma1}.

 \begin{figure}
	\begin{center}
		
		\begin{tikzpicture}
		\foreach \i in {-1,0,...,12}
		{		
			\node[draw, circle,fill=white,minimum size=6pt, inner sep=0pt]
			(v\i) at (\i,0) {};	
		}

\node[draw, circle,fill=black,minimum size=6pt, inner sep=0pt] at (2,0) {};	
\node[draw, circle,fill=black,minimum size=6pt, inner sep=0pt] at (3,0) {};	
\node[draw, circle,fill=black,minimum size=6pt, inner sep=0pt] at (6,0) {};
\node[draw, circle,fill=black,minimum size=6pt, inner sep=0pt] at (8,0) {};
\node[draw, circle,fill=black,minimum size=6pt, inner sep=0pt] at (9,0) {};
\node[draw, circle,fill=black,minimum size=6pt, inner sep=0pt] at (10,0) {};

		\foreach \i in {-1,0,...,8,9}
		{	
			\pgfmathint{\i + 1}
			\edef\j{\pgfmathresult}
			\draw[->] (v\i) edge (v\j);
		}
		
		\foreach \i in {10,11}
		{	
			\pgfmathint{\i + 1}
			\edef\j{\pgfmathresult}
			\draw[<-] (v\i) edge (v\j);
		}
		
		\node  at (-1,-0.5) {$v_{0}$};
		\node  at (0,-0.5) {$v_{1}$};
		\node at (2,-0.5) {$v_{h_1}$};
		\node at (3,-0.5) {$v_{h_2}$};
		\node at (6,-0.5) {$v_{h_3}$};
		\node at (8,-0.5) {$v_{h_4}$};
		\node at (9,-0.5) {$v_{h_5}$};
		\node at (10,-0.5) {$v_{h_6}$};
		\node at (12,-0.5) {$v_{13}$};
		\path[draw,thick,black!60!green]
		
		(v-1) edge node[name=la,pos=0.5, above] {\color{blue} $13$} (v0)
		(v0) edge node[name=la,pos=0.5, above] {\color{blue} $1$} (v1)
		(v1) edge node[name=la,pos=0.5, above] {\color{blue}$11$} (v2)
		(v2) edge node[name=la,pos=0.5, above] {\color{blue} $10$} (v3)
		(v3) edge node[name=la,pos=0.5, above] {\color{blue} $8$} (v4)
		(v4) edge node[name=la,pos=0.5, above] {\color{blue} $3$} (v5)
		(v5) edge node[name=la,pos=0.5, above] {\color{blue} $9$} (v6)
		(v6) edge node[name=la,pos=0.5, above] {\color{blue} $4$} (v7)
		(v7) edge node[name=la,pos=0.5, above] {\color{blue} $7$} (v8)
		(v8) edge node[name=la,pos=0.5, above] {\color{blue} $6$} (v9)
		(v9) edge node[name=la,pos=0.5, above] {\color{blue} $5$} (v10)
		(v10) edge node[name=la,pos=0.5, above] {\color{blue} $2$} (v11)
		(v11) edge node[name=la,pos=0.5, above] {\color{blue} $12$} (v12);
		
		\node  at (6,-1.5) {Labeling of $\overrightarrow{P}$ when $\ell$ is even and $s=1$};
		
		\end{tikzpicture}
		
		\smallskip
		
		\begin{tikzpicture}
		\foreach \i in {-1,0,...,12}
		{
			
			\node[draw, circle,fill=white,minimum size=6pt, inner sep=0pt]
			(v\i) at (\i,0) {};
			
		}

\node[draw, circle,fill=black,minimum size=6pt, inner sep=0pt] at (3,0) {};	
\node[draw, circle,fill=black,minimum size=6pt, inner sep=0pt] at (4,0) {};	
\node[draw, circle,fill=black,minimum size=6pt, inner sep=0pt] at (6,0) {};
\node[draw, circle,fill=black,minimum size=6pt, inner sep=0pt] at (7,0) {};
\node[draw, circle,fill=black,minimum size=6pt, inner sep=0pt] at (10,0) {};

		\foreach \i in {-1,0,...,8,9}
		{	
			\pgfmathint{\i + 1}
			\edef\j{\pgfmathresult}
			\draw[->] (v\i) edge (v\j);
		}

		\foreach \i in {10,11}
		{	
			\pgfmathint{\i + 1}
			\edef\j{\pgfmathresult}
			\draw[<-] (v\i) edge (v\j);
		}
		
		\node  at (-1,-0.5) {$v_{0}$};
		\node  at (0,-0.5) {$v_{1}$};
		\node at (4,-0.5) {$v_{h_2}$};
		\node at (3,-0.5) {$v_{h_1}$};
		\node at (6,-0.5) {$v_{h_3}$};
		\node at (7,-0.5) {$v_{h_4}$};
		\node at (10,-0.5) {$v_{h_5}$};
		\node at (12,-0.5) {$v_{13}$};
		\path[draw,thick,black!60!green]
		
		(v-1) edge node[name=la,pos=0.5, above] {\color{blue} $13$} (v0)
		(v0) edge node[name=la,pos=0.5, above] {\color{blue} $1$} (v1)
		(v1) edge node[name=la,pos=0.5, above] {\color{blue}$11$} (v2)
		(v2) edge node[name=la,pos=0.5, above] {\color{blue} $10$} (v3)
		(v3) edge node[name=la,pos=0.5, above] {\color{blue} $9$} (v4)
		(v4) edge node[name=la,pos=0.5, above] {\color{blue} $3$} (v5)
		(v5) edge node[name=la,pos=0.5, above] {\color{blue} $8$} (v6)
		(v6) edge node[name=la,pos=0.5, above] {\color{blue} $7$} (v7)
		(v7) edge node[name=la,pos=0.5, above] {\color{blue} $6$} (v8)
		(v8) edge node[name=la,pos=0.5, above] {\color{blue} $4$} (v9)
		(v9) edge node[name=la,pos=0.5, above] {\color{blue} $5$} (v10)
		(v10) edge node[name=la,pos=0.5, above] {\color{blue} $2$} (v11)
		(v11) edge node[name=la,pos=0.5, above] {\color{blue} $12$} (v12);
		
		\node  at (6,-1.5) {Labeling of $\overrightarrow{P}$ when $\ell$ is even and $s=2$};
		
		\end{tikzpicture}
		
		\smallskip
		
		\begin{tikzpicture}
		\foreach \i in {0,...,13}
		{		
			\node[draw, circle,fill=white,minimum size=6pt, inner sep=0pt]
			(v\i) at (\i,0) {};
			
		}

\node[draw, circle,fill=black,minimum size=6pt, inner sep=0pt] at (3,0) {};	
\node[draw, circle,fill=black,minimum size=6pt, inner sep=0pt] at (5,0) {};	
\node[draw, circle,fill=black,minimum size=6pt, inner sep=0pt] at (6,0) {};
\node[draw, circle,fill=black,minimum size=6pt, inner sep=0pt] at (7,0) {};
\node[draw, circle,fill=black,minimum size=6pt, inner sep=0pt] at (10,0) {};

		\foreach \i in {0,...,9}
		{	
			\pgfmathint{\i + 1}
			\edef\j{\pgfmathresult}
			\draw[->] (v\i) edge (v\j);
		}

		\foreach \i in {10,11,12}
		{	
			\pgfmathint{\i + 1}
			\edef\j{\pgfmathresult}
			\draw[<- ] (v\i) edge (v\j);
		}
		
		\node  at (0,-0.5) {$v_{0}$};
		\node at (3,-0.5) {$v_{h_1}$};
		\node at (5,-0.5) {$v_{h_2}$};
		\node at (6,-0.5) {$v_{h_3}$};
		\node at (7,-0.5) {$v_{h_4}$};
		\node at (10,-0.5) {$v_{h_5}$};
		
		\node at (13,-0.5) {$v_{13}$};
		\path[draw,thick,black!60!green]
		
		(v0) edge node[name=la,pos=0.5, above] {\color{blue} $13$} (v1)
		(v1) edge node[name=la,pos=0.5, above] {\color{blue}$1$} (v2)
		(v2) edge node[name=la,pos=0.5, above] {\color{blue} $11$} (v3)
		(v3) edge node[name=la,pos=0.5, above] {\color{blue} $3$} (v4)
		(v4) edge node[name=la,pos=0.5, above] {\color{blue} $9$} (v5)
		(v5) edge node[name=la,pos=0.5, above] {\color{blue} $8$} (v6)
		(v6) edge node[name=la,pos=0.5, above] {\color{blue} $7$} (v7)
		(v7) edge node[name=la,pos=0.5, above] {\color{blue} $5$} (v8)
		(v8) edge node[name=la,pos=0.5, above] {\color{blue} $4$} (v9)
		(v9) edge node[name=la,pos=0.5, above] {\color{blue} $6$} (v10)
		(v10) edge node[name=la,pos=0.5, above] {\color{blue} $10$} (v11)
		(v11) edge node[name=la,pos=0.5, above] {\color{blue} $2$} (v12)
		(v12) edge node[name=la,pos=0.5, above] {\color{blue} $12$} (v13);
		
		\node  at (7,-1.5) {Labeling of $\overrightarrow{P}$ when $\ell$ is odd and $s=0$};
		\end{tikzpicture}

		\smallskip
		
		\begin{tikzpicture}
		\foreach \i in {0,...,13}
		{			
			\node[draw, circle,fill=white,minimum size=6pt, inner sep=0pt]
			(v\i) at (\i,0) {};			
		}

\node[draw, circle,fill=black,minimum size=6pt, inner sep=0pt] at (4,0) {};	
\node[draw, circle,fill=black,minimum size=6pt, inner sep=0pt] at (5,0) {};	
\node[draw, circle,fill=black,minimum size=6pt, inner sep=0pt] at (6,0) {};
\node[draw, circle,fill=black,minimum size=6pt, inner sep=0pt] at (7,0) {};
\node[draw, circle,fill=black,minimum size=6pt, inner sep=0pt] at (10,0) {};

		\foreach \i in {0,...,9}
		{	
			\pgfmathint{\i + 1}
			\edef\j{\pgfmathresult}
			\draw[->] (v\i) edge (v\j);
		}

		\foreach \i in {10,11,12}
		{	
			\pgfmathint{\i + 1}
			\edef\j{\pgfmathresult}
			\draw[<- ] (v\i) edge (v\j);
		}
		
		\node  at (0,-0.5) {$v_{0}$};
		\node at (4,-0.5) {$v_{h_1}$};
		\node at (5,-0.5) {$v_{h_2}$};
		\node at (6,-0.5) {$v_{h_3}$};
		\node at (7,-0.5) {$v_{h_4}$};
		\node at (10,-0.5) {$v_{h_5}$};
		
		\node at (13,-0.5) {$v_{13}$};
		\path[draw,thick,black!60!green]
		
		(v0) edge node[name=la,pos=0.5, above] {\color{blue} $13$} (v1)
		(v1) edge node[name=la,pos=0.5, above] {\color{blue}$1$} (v2)
		(v2) edge node[name=la,pos=0.5, above] {\color{blue} $11$} (v3)
		(v3) edge node[name=la,pos=0.5, above] {\color{blue} $9$} (v4)
		(v4) edge node[name=la,pos=0.5, above] {\color{blue} $8$} (v5)
		(v5) edge node[name=la,pos=0.5, above] {\color{blue} $7$} (v6)
		(v6) edge node[name=la,pos=0.5, above] {\color{blue} $6$} (v7)
		(v7) edge node[name=la,pos=0.5, above] {\color{blue} $4$} (v8)
		(v8) edge node[name=la,pos=0.5, above] {\color{blue} $3$} (v9)
		(v9) edge node[name=la,pos=0.5, above] {\color{blue} $5$} (v10)
		(v10) edge node[name=la,pos=0.5, above] {\color{blue} $10$} (v11)
		(v11) edge node[name=la,pos=0.5, above] {\color{blue} $2$} (v12)
		(v12) edge node[name=la,pos=0.5, above] {\color{blue} $12$} (v13);
		
		\node  at (7,-1.5) {Labeling of $\overrightarrow{P}$ when $\ell$ is odd and $s=1$};
		\end{tikzpicture}	
		\caption{Labelings of $P$ satisfy requirements of Lemma~\ref{lemma1}}
		\label{fig1}
		
	\end{center}
	
\end{figure}
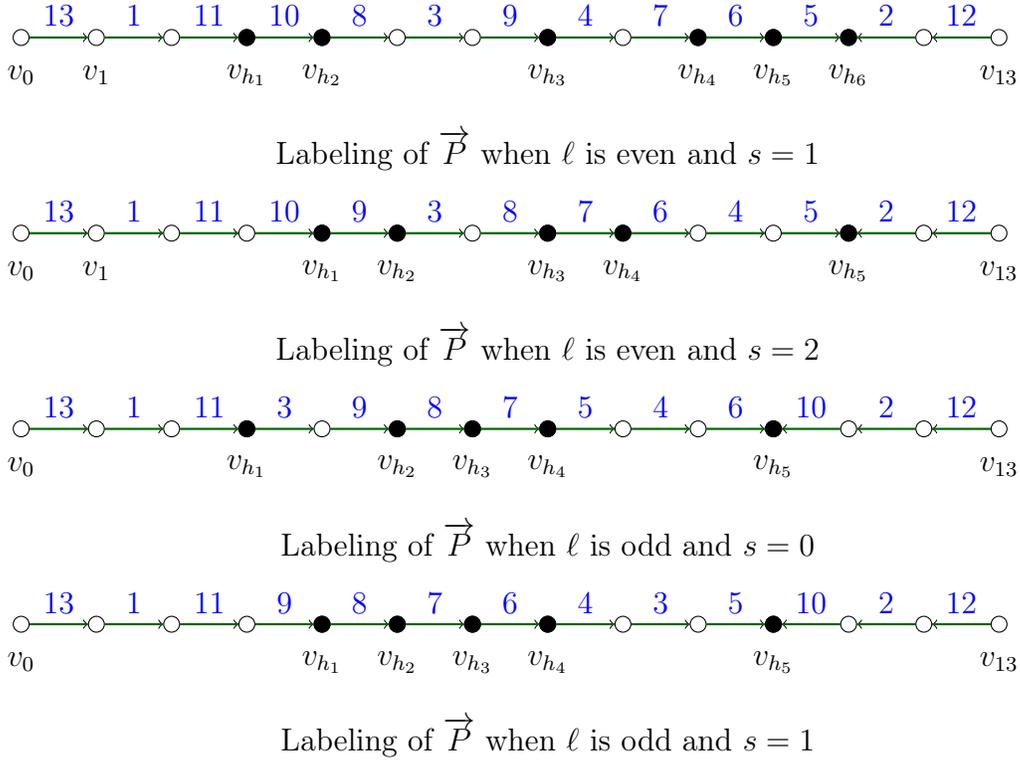

For the vertex $v_{h_t}$, by the orientation of $P$,  we have $v_{h_t-1}v_{h_t}\in A(\overrightarrow{P})$ and $ v_{h_t+1}v_{h_t}\in A(\overrightarrow{P})$, thus $s(v_{h_t})>0$.
By Step 2 and step 3, $\tau(v_{h_1-1}v_{h_1})>\tau(v_{h_1}v_{h_1+1})$. Therefore,
$s(v_{h_1})=\tau(v_{h_1-1}v_{h_1})-\tau(v_{h_1}v_{h_1+1})>0$.
For each $v_{h_i}\in U\setminus \{v_{h_1}, v_{h_t}\}$, by Step 3,
we always have 	$\tau(v_{h_i-1}v_{h_i})>\tau(v_{h_i}v_{h_i+1})$. Therefore,
$s(v_{h_i})=\tau(v_{h_i-1}v_{h_i})-\tau(v_{h_i}v_{h_i+1})>0$. This proves (i).

Statement (ii) is obvious as for each $v_i\in V(P)\setminus U$ with $i\ne 0, m$,
$1\le |s(v_i)|=|\tau(v_{i-1}v_i)-\tau(v_iv_{i+1})|\le m-1$, $s(v_0)=-m$ and $s(v_m)=-(m-1)$.

We now show (iii). Let $v_i,v_j \in V(P)\setminus U$ with $i<j$.
 By Step 1 and Step 2,  for each $k\in [0,t]$,  if $v_i,v_j \in V(P_k)\setminus U$,
we have  $|s(v_i)|< |s(v_j)|$ unless $k=0$ or $k=t-1$ in the three exceptional cases in Step 4.
In these cases, $|s(v_i)|> |s(v_j)|$.  Also, by Step 3 and Step 4,  if $v_i,v_j \in V(P)\setminus (U\cup V(P_0)) $ such that $v_i$ and $v_j$ are from distinct subpaths,
we have  $|s(v_i)|> |s(v_j)|$ unless  $v_j$ is from $V(P_{t})$.
In this case, $|s(v_i)|< |s(v_j)|$.
Thus, we are only left to show that $s(v_i)\ne s(v_j)$ when  $v_i\in V(P_0)$ and $v_j\in V(P)\setminus (U\cup V(P_0)) $. Note that by Step 2, the set of vertex-sums on vertices from $V(P_0)$ is as follows:

$$
\begin{cases}
	\{-m, m-1, -(m-3), m-5, \cdots, m-2\ell+3,\\ -(m-2\ell+1)\}, & \text{if  $\ell$ is even and $s=1$};\\
\{-m, m-1, -(m-3), m-5, \cdots, m-2\ell+3,\\ -(m-2\ell+1), m-2\ell-1, -(m-2\ell-2), m-2\ell-3,\\ \cdots,-(m-2\ell-s+1)\}, & \text{if  $\ell$ is even, $s$ is odd and $s\ge 3$};\\
\{-m,2\}, & \text{if  $\ell=2$  and $s=0$};\\
\{-m, m-1, -(m-3), m-5, \cdots, m-2\ell+7,\\ -(m-2\ell+5), 2\}, & \text{if  $\ell\ge 4$ is even and $s=0$};\\
\{-m, m-1, -(m-3), m-5, \cdots, m-2\ell+3, \\-(m-2\ell+1), 1\}, & \text{if  $\ell$ is even and $s=2$};\\
 \{-m, m-1, -(m-3), m-5, \cdots, m-2\ell+3,\\ -(m-2\ell+1), m-2\ell-1, -(m-2\ell-2), m-2\ell-3,\\ \cdots, m-2\ell-s+3,-(m-2\ell-s+2),1\}, & \text{if  $\ell$ is even, $s$ is even and $s\ge 4$};\\
 \{-m, m-1, -(m-3), m-5, \cdots, m-2\ell+5, \\-(m-2\ell+3)\}, & \text{if  $\ell$ is odd and $s=0$};\\
  \{-m, m-1, -(m-3), m-5, \cdots, m-2\ell+5,\\ -(m-2\ell+3), m-2\ell+1, -(m-2\ell-1)\}, & \text{if  $\ell$ is odd and $s=2$};\\
  \{-m, m-1, -(m-3), m-5, \cdots, m-2\ell+5, \\ -(m-2\ell+3),m-2\ell+1, -(m-2\ell-1), m-2\ell-2,\\-(m-2\ell-3), \cdots, m-2\ell-s+2,-(m-2\ell-s+1)\}, & \text{if  $\ell$ is odd, $s$ is even and $s\ge 4$};\\
 \{-m,2\}, & \text{if  $\ell=1$ and $s=1$};\\
\{-m, m-1, -(m-3), m-5, \cdots, m-2\ell+5, \\-(m-2\ell+3),2\}, & \text{if  $\ell\ge 3$ is odd and $s=1$};\\
\{-m, m-1, -(m-3), m-5, \cdots, m-2\ell+5, \\-(m-2\ell+3), m-2\ell+1, -(m-2\ell-1), 1\}, & \text{if  $\ell$ is odd and $s=3$};\\
  \{-m, m-1, -(m-3), m-5, \cdots, m-2\ell+5,\\ -(m-2\ell+3), m-2\ell+1, -(m-2\ell-1), m-2\ell-2,\\-(m-2\ell-3), \cdots,-(m-2\ell-s+2),1\}, & \text{if  $\ell$ is odd, $s$ is odd and $s\ge 5$}.
	\end{cases}
	$$

We consider two cases regarding where the vertex $v_j$ is.

{\bf \noindent Case 1:   $v_j\in V(P_t) $.}

By Step 1, if $\ell $ is even, then  $s(v_j)\in \{-(m-1), m-3, -(m-5), m-7, \cdots, -(m-2\ell+3), m-2\ell+1\}$; and
if $\ell$ is odd, then $s(v_j)\in \{-(m-1), m-3, -(m-5), m-7, \cdots, (m-2\ell+3), -(m-2\ell+1)\}$.
Since the value $s(v_i)=-(m-2\ell+1)$ is only achieved when $\ell$ is even, we see that
 $s(v_j)\ne s(v_i)$ in either case.

{\bf \noindent Case 2:   $v_j\in V(P_k) $ for some $k\in [t-1]$.}

By Steps 3 and 4,  $|s(v_j)|\le m-2\ell -s<m-2\ell-s+1$.
Thus, $s(v_j)\ne s(v_i)$ if $ |s(v_i)|\ge m-2\ell-s+1$.
Hence, we assume that $s(v_i)\in \{1,2\}$, which in turn implies that $h_1$ is even.
By the labeling of the three exceptional cases in Step 4, when $s(v_i)=1$,  we have $|s(v_j)|\ge 2$ or $s(v_j)=-1$;
and when $s(v_i)=2$,  we have $|s(v_j)|\ge 3$,  $s(v_j)=1$
or $s(v_j)=-2$. In any cases, we get $s(v_j)\ne s(v_i)$.
This finishes the proof for (iii).
\end{proof}

\section{Proof of Theorem \ref{Thm}}
Let $T$ be a lobster. A \emph{leaf} of $T$ is a vertex of degree one in $T$,
a \emph{spine} of $T$ is a longest path $P$ of $T$, and a  \emph{leg} of $T$ is a path of $T$
that joins one leaf from $T-V(P)$ and one vertex from $P$.

 Let $T$ be a lobster with $m$ edges, and $P=v_0v_1\cdots v_p$ be a spine of $T$. If $T=P$, then $T$ has an antimagic orientation by Lemma \ref{lemma3}. So we assume $T\neq P$. Let
  $U=\{v_{h_1},v_{h_2},\cdots,v_{h_t}\}\subseteq V(P)$ be the set of vertices of degree at least three in $T$, where $h_1<h_2<\cdots<h_t$.
  Since $T$ is not a path, $|U|\ge 1$.
  Define
   $$X=\{u\in V(T)\setminus V(P)\mid u\ \mbox{is\ adjacent\ to\ a\ vertex\ in}\ U\}, \quad \text{and} \quad Y=V(T)\setminus(X\cup V(P)).$$  By the definition of a lobster,  for each vertex $u\in Y$, $d_T(u)=1$. We define an antimagic orientation $(\overrightarrow{T},\tau)$ of $T$ in four steps.
 Note that $V(P)\cup X\cup Y=V(T)$.
   \begin{enumerate}[Step 1]
   	\item Orient the edges in $E(P)$ and label them.
   	
   	By Lemma 1, $P$ has an orientation $\overrightarrow{P}$ and a labeling $\tau_1$ using numbers $1,2,\cdots,p$ satisfies the following conditions:
   	\begin{enumerate}[(i)]
   		\item $s_{\overrightarrow{P}}(v)\geq 1$ for any vertex $v\in U$;
   		\item $1\leq |s_{\overrightarrow{P}}(v)|\leq p$ for any vertex $v\in V(P)\setminus U$; and
   		\item for any two  distinct vertices $u,v\in V(P)\setminus U$, we have that $s_{\overrightarrow{P}}(u)\neq s_{\overrightarrow{P}}(v)$.
   	\end{enumerate}

   	\item Orient the remaining edges.
   	
   	For each edge $uv\in E(T)$ with $u\in X$, orient $uv$ from $u$ to $v$.  That is, all edges incident to a vertex in $X$ are oriented  towards either $P$ or $Y$.

   	Let $M_1$ be a maximum matching of $T$ using edges   between $X$ and $V(P)$, and $M_2$ be a maximum matching of $T$ using edges   between $X$ and $Y$.  In other words,  $M_1$
   	is a matching of size $|U|$ that saturates all vertices in $U$ and $M_2$
   	is a matching of size $|X|$ that saturates all vertices in $X$.
   	Clearly, $M_1$ and $M_2$ are edge-disjoint.

\item Label  arcs in $\partial(X)\backslash (M_1\cup M_2)$.
   	
Label the arcs in $A(X,Y\setminus V(M_2))$ using numbers in $[p+|M_2|+1,p+|Y|]$  one by one arbitrarily. (We reserve numbers in $[p+1, p+|M_2|]$ for edges in $M_2$. )

Let $$X_1=\{u\in X\mid d_{T}(u)=1\} \quad \text{and} \quad |A(X_1, V(P)\cap V(M_1))|=r.$$
 Label arcs in $A(X_1, V(P)\setminus V(M_1))$ with numbers in $[p+|Y|+1,p+|Y|+|X_1|-r]$ one by one arbitrarily. Label arcs in $A(X\setminus (X_1\cup V(M_1)), V(P))$ with numbers in $[p+|Y|+|X_1|-r+1+|M_1|,m]$ one by one arbitrarily (numbers in $[p+|Y|+|X_1|-r+1,p+|Y|+|X_1|-r+|M_1|]$ are reserved for edges in $M_1$).
Denote the current labeling  by $\tau_2$ and the partial vertex-sum of a vertex $x\in V(T)$  by $s_2(x)$.

\item  Label arcs in $M_1\cup M_2$.

Assume that $s_2(v_{i_1})\leq s_2(v_{i_2})\leq \cdots \leq s_2(v_{i_t})$, where $\{i_1,i_2,\cdots,i_t\}=\{h_1, h_2, \cdots, h_t\}$. Label arcs in $M_1$ incident with $v_{i_1},v_{i_2},\cdots,v_{i_t}$ with $p+|Y|+|X_1|-r+1,\cdots,p+|Y|+|X_1|-r+|M_1|$, respectively in this order. The resulting labeling is denoted by $\tau_3$ and the partial  vertex-sum of a vertex $x\in V(T)$ is denoted by $s_3(x)$. By the way of assigning labels to edges in $M_1$, we have
\begin{equation}\label{ine1}
s_3(v_{i_1})<s_3(v_{i_2})< \cdots <s_3(v_{i_t}).
\end{equation}

Let $X\setminus X_1=\{x_1,x_2,\cdots,x_q\}$.
Assume, without loss of generality,  that  $s_3(x_1)\leq s_3(x_2)\leq \cdots \leq s_3(x_q)$. Label the arcs in $M_2$ incident with $x_1,x_2,\cdots,x_q$ with $p+1,p+2,\cdots,p+|M_2|$,  respectively in this order. The resulting labeling is denoted by $\tau_4$ and the vertex-sum of a vertex $x\in V(T)$ is denoted by $s_4(x)$.
By the way of assigning labels to edges in $M_2$, we have
\begin{equation}\label{ine2}
s_4(x_1)< s_4(x_2)< \cdots <s_4(x_q).
\end{equation}
\end{enumerate}

It is easy to see that $\tau_4$ is a bijection from $A(\overrightarrow{T})$ to $[m]$. We  show that $\tau_4$ is an antimagic labeling of $\overrightarrow{T}$ in the following.
Let $u,v\in V(\overrightarrow{T})$ be any two distinct vertices.

Note that for every $x\in V(P)\setminus U$, $s_4(x)= s_{\overrightarrow{P}}(x)$.
If $u,v\in V(P)\setminus U$, then $s_4(u)\ne s_4(v)$ by Step 1.
 If $u,v\in U$, then $s_4(u)\neq s_4(v)$ by~\eqref{ine1}.
Thus we assume  $u\in V(P)\setminus U$ and $v\in U$. Then $|s_4(u)|\le p$ and $s_4(v)>p$. Therefore, if $u,v\in V(P)$, it holds that
$s_4(u)\neq s_4(v)$.

Consider now that  $u,v\in X$.
If $u,v\in X_1$,  then $s_4(u)\neq s_4(v)$ as $u$ and $v$ are leaves of $T$.
If  $u,v\in X\setminus X_1$, then $s_4(u)\neq s_4(v)$ by~\eqref{ine2}. Thus we assume that
$u\in X_1$ and $v\in X\setminus X_1$.
By Steps 3 and 4,   labels assigned to edges in $A( X_1, V(P))$ is a subset of $[p+|Y|+1, p+|Y|+|X_1|-r+|M_1|]. $
Thus $|s_4(u)|\leq p+|Y|+|X_1|-r+|M_1|$. Every vertex in $X\setminus X_1$ is adjacent in $T$ to a
vertex from $Y$ and is adjacent to a vertex from $V(P)$. By Steps 3 and 4, the smallest label that is assigned to edges in $A(X,Y)$
is $p+1$, and the smallest label that is assigned to edges in $A(X\setminus X_1, V(P))$
is $p+|Y|+|X_1|-r+1$.
Therefore $|s_4(v)|\geq p+1+p+|Y|+|X_1|-r+1>p+|Y|+|X_1|-r+|M_1|+2$, as  $p>|M_1|$. Thus $s_4(v) \ne s_4(u)$.  Therefore if $u,v\in X$, then $s_4(u)\neq s_4(v)$.

Since any vertex in $Y$ is a leaf of $T$,  if  $u,v\in Y$, then it naturally holds that $s_4(u)\neq s_4(v)$.

Thus  we only need to prove that  when $u$ and
$v$ are from two distinct sets among  $V(P), X$, and $Y$, it holds that $s_4(u)\neq s_4(v)$.
This follows by the following arguments.
For every $y\in Y$,  by Step 2, the arc incident to $y$ is entering $y$. Thus, by Steps  3 and 4, $s_4(y)\in [p+1,p+|Y|]$.
For every $x\in X$,  by Step 2, the arcs incident to $x$ are leaving $x$. Thus, by Steps  3 and 4, $s_4(x)<0$ and  $|s_4(x)|\ge p+|Y|+1 $.
For every $z\in V(P)$,  by Step 1, we have either
$|s_4(z)|\le p$ or $s_4(z)\ge p+|Y|+1$.
Therefore $\tau_4$ is an antimagic labeling of $\overrightarrow{T}$.
\section{Open problem}

In this paper, we showed that every lobster admits an antimagic orientation. Since every bipartite antimagic graph
$G$ admits an antimagic orientation, it is natural to ask that whether lobsters are antimagic.
We believe this is true.

\begin{conj} Every lobster admits an antimagic labeling.
\end{conj}

\bibliography{Bibifile1}

\end{document}